\documentclass[12pt,oneside,reqno,bibliography=totoc]{scrartcl}
\usepackage{kucuk_cihan}

\newcommand {\czer}[2] {\mathbf{C}_{#1}(#2)}

\newlist{enume}{enumerate}{1}
\setlist[enume]{label={\arabic*.}}

\begin{document}

\title{Notes on splitting fields}
\author{C\.{i}han Bahran}

\date{}


\maketitle

\begin{conv}
 Throughout this document, we fix a field $k$ and a (unital associative) $k$-algebra $A$.
\end{conv}


Let $k \subseteq F$ be a field extension. Then we naturally have a restriction functor $\alg{F} \rarr \alg{k}$ which has the left adjoint $F \otimes_{k} - \colon \alg{k} \rarr \alg{F}$. The unit of this adjunction on $A$ is the $k$-algebra homomorphism 
\begin{align*}
 \varphi \colon A &\rarr F \otimes_{k} A \\
 a &\mapsto 1 \otimes a
\end{align*}
Note that $\varphi$ is injective. Now $\varphi$ yields a restriction functor $\varphi_{*} \colon \lMod{(F \otimes_{k} A)} \rarr \lMod{A}$. $\varphi_{*}$ has a left adjoint, say $\varphi^{*} \colon \lMod{A} \rarr \lMod{(F \otimes_{k} A)}$.
\begin{prop}
 $\varphi^{*}$ is naturally isomorphic to $F \otimes_{k} -$.
\end{prop}
\begin{proof}
 The functor $F \otimes_{k} -$ here is not quite the same with the one between the categories of algebras. Note that both $A$ and $F$ are $k$-algebras. Hence if we have a left $A$-module $M$ and an $F$-module V, then $V \otimes_{k} M$ has a natural $F \otimes_{k} A$-module structure. Taking $V = F$, we see that $F \otimes_{k} M$ is a left $F \otimes_{k} A$-module. So the assignment 
\begin{align*}
 F \otimes_{k} - \colon
 \lMod{A} &\rarr \lMod{(F \otimes_{k} A)} \\
 M &\mapsto F \otimes_{k} M
\end{align*}
is well-defined on objects. For morphisms, there is a natural candidate: if $f \colon M \rarr M'$ is an $A$-linear map, then let $(F \otimes_{k} -)(f) = \id_{F} \otimes f \colon F \otimes_{k} M \rarr F \otimes_{k} M'$. We have
\begin{align*}
 (\id_{F} \otimes f) \left( (\alpha \otimes r)\cdot(\beta \otimes m) \right)
 &= (\id_{F} \otimes f)(\alpha\beta \otimes rm) \\
 &= \alpha\beta \otimes f(rm) \\
 &= \alpha\beta \otimes rf(m) \\
 &= (\alpha \otimes r)\cdot(\beta \otimes f(m)) \\
 &= (\alpha \otimes r)\cdot(\id_{F} \otimes f)(\beta \otimes m)
\end{align*}
for any $\alpha,\beta \in F$, $r \in A$, $m \in M$, that is, $\id_{F} \otimes f$ is really a left $F \otimes_{k} A$-module homomorphism. So we have well-defined assignments on both the objects and morphisms for $F \otimes_{k} -$ which satisfy the functor axioms.

Now let's move on to the adjunction with $\varphi_{*}$. Given a left $A$-module $M$ and a left $F \otimes_{k} A$-module $N$, we seek an isomorphism
\begin{align*}
 \Hom_{A}(M, \varphi_{*}(N)) \cong \Hom_{F \otimes_{k} A}(F \otimes_{k} M, N)
\end{align*}
that is natural in $M$ and $N$.
(Note that $N$ is an $A$-module via $\varphi\colon r \mapsto 1_{F} \otimes r$)\\
We already have a natural isomorphism
\begin{align*}
 \eta: \Hom_{k}(M,N) &\rarr \Hom_{F}(F \otimes_{k} M, N) \\
 f &\mapsto \eta(f): (\alpha \otimes m) \mapsto \alpha f(m) = (\alpha \otimes 1_{A})f(m)
\end{align*}
by the usual theory of extension of scalars since $N$ is an $F$-vector space via $\alpha n = (\alpha \otimes 1_{A})n $. We need to check that if $f$ is $A$-linear then $\eta(f)$ is $F \otimes_{k} A$-linear. Indeed, the routine check
\begin{align*}
 \eta(f)((\beta \otimes r) \cdot (\alpha \otimes m)) 
 &= \eta(f)(\beta\alpha \otimes rm)\\
 &= (\beta\alpha \otimes 1_{A}) f(rm) \\
 &= (\beta\alpha \otimes 1_{A}) (1_{F} \otimes r) f(m) \\
 &= (\beta\alpha \otimes r)f(m) \\
 &= (\beta \otimes r)(\alpha \otimes 1_{A})f(m) \\
 &= (\beta \otimes r)\cdot \eta(f)(\alpha \otimes m)
\end{align*}
verifies so. The inverse of $\eta$ is given by
\begin{align*}
 \theta \colon \Hom_{F}(F \otimes_{k} M, N) &\rarr \Hom_{k}(M,N) \\
 g &\mapsto \theta(g) \colon m \mapsto g(1_{F} \otimes m)
\end{align*}
And if $g$ is $F \otimes_{k} A$-linear, then we have
\begin{align*}
 \theta(g)(rm) 
 &= g(1_{F} \otimes rm) \\
 &= g((1_{F} \otimes r) \cdot (1_{F} \otimes m)) \\
 &= (1_{F} \otimes r) \cdot g(1_{F} \otimes m) \\
 &= r \cdot g(1_{F} \otimes m) \\
 &= r \cdot \theta(g)(m) \, ,
\end{align*}
that is, $\theta(g)$ is $A$-linear. Thus $\eta$ and $\theta$ restrict to the desired isomorphism and they are already natural in $M$ and $N$.
\end{proof}

Let's take a step back and look at what we've done. Everything stemmed from the field extension $k \subseteq F$ which gives us a way to induce $k$-algebras to $F$-algebras and to induce modules over a fixed $k$-algebra $A$ to modules over the induced $F$-algebra. We called this induction functor $\varphi^{*}$ above and saw that it is naturally isomorphic to $F \otimes_{k} -$. I find it more suggestive to write $(-)^{F}$ for this functor and also $A^{F}$ for the induced algebra $F \otimes_{k} A$.
\begin{rem} \label{zero}
 Let $M$ be a left $A$-module. Then 
 \(\dim_{F} M^{F} = \dim_{k} M \), 
and in particular $M = 0$ if and only if $M^{F} = 0$.
\end{rem}
\begin{prop} \label{topdown}
 Let $M$ be a left $A$-module and $k \subseteq F$ a field extension. 
\begin{birki}
 \item If $M^{F}$ is indecomposable as a left $A^{F}$-module, then $M$ is indecomposable.
 \item If $M^{F}$ is simple as a left $A^{F}$-module, then $M$ is simple.
 \end{birki}

\end{prop}
\begin{proof}
 Better to use contrapositives. 
 
 For (1), suppose $M$ is not indecomposable, that is, either $M=0$ or $M = U \oplus V$ with $U,V \neq 0$ being $A$-submodules of $M$. In the former case $M^{F} = 0$. In the latter case, we have an isomorphism
 \[
  M^{F} \cong U^{F} \oplus V^{F}
 \] 
 of left $A^{F}$-modules where $U^{F}, V^{F} \neq 0$ by Remark \ref{zero}. In any case, $M^{F}$ is not decomposable. 
 
 For (2), suppose $M$ is not simple, that is, either $M=0$ or there exists a short exact sequence
 \[
 0 \to N \to M \to Q \to 0
 \]
 of left $A$-modules with $N,Q \neq 0$. In the former case $M^{F} = 0$. In the latter case, as $(-)^{F}$ is an exact functor, we get a short exact sequence
 \[
 0 \to N^{F} \to M^{F} \to Q^{F} \to 0
 \]
 of left $A^{F}$-modules with $N^{F},Q^{F} \neq 0$ by Remark \ref{zero}. In any case, $M^{F}$ is not simple. 
\end{proof}

Given left $A$-modules $M$ and $N$, via the functor $(-)^{F} \colon \lMod{A} \rarr \lMod{A^{F}}$, we have a map
\begin{align*}
 \Hom_{A}(M,N) \rarr \Hom_{A^{F}}(M^{F},N^{F}) \, .
\end{align*}
This map is actually $k$-linear. Since the right hand side is an $F$-vector space, via the adjunction of $(-)^{F}$ with the restriction, we get an $F$-linear map
\begin{align*}
\theta \colon \left[ \Hom_{A}(M,N) \right]^{F} \rarr \Hom_{A^{F}}(M^{F},N^{F}) \, .
\end{align*}

\begin{lem}\label{iso}
For every pair of left $A$-modules $M,N$, the natural $F$-linear map 
\begin{align*}
\theta \colon \left[ \Hom_{A}(M,N) \right]^{F} \rarr \Hom_{A^{F}}(M^{F},N^{F})
\end{align*} 
defined above is injective. If moreover $\dim_{k} M < \infty$, then $\theta$ is an isomorphism.
\end{lem}
\begin{proof}
 Let $\{ \alpha_{i}: i \in I \}$ be a basis of $F$ over $k$. Note that an arbitrary element of $\left[ \Hom_{A}(M,N)\right]^{F} = F \otimes_{k} \Hom_{A}(M,N)$ is of the form 
\begin{align*}
 \sum \alpha_{i} \otimes f_{i}
\end{align*}
where $\alpha_{i} \in F$ and $f_{i} \colon M \rarr N$ is $A$-linear. If such an element is in $\ker \theta$, then
\begin{align*}
 0 &= \sum \alpha_{i} (f_{i})^{F} \\
 &= \sum \alpha_{i} (\id_{F} \otimes f_{i}) \, .
\end{align*}
So for every $m \in M$, we have
\begin{align*}
 0 
 &= \left( \sum \alpha_{i} (\id_{F} \otimes f_{i}) \right)(1 \otimes m) \\
 &= \sum \alpha_{i} \left( 1 \otimes f_{i}(m) \right) \\
 &= \sum \alpha_{i} \otimes f_{i}(m) \, .
\end{align*}
This equality happens inside $N^{F}$ and we have
\begin{align*}
 N^{F} 
 = F \otimes_{k} N 
 &\cong \left( \bigoplus_{i \in I} k \right) \otimes_{k} N \\
 &\cong \bigoplus_{i \in I} \left( k \otimes_{k} N \right)
 \cong \bigoplus_{i \in I} N
\end{align*}
where the isomorphism from the beginning to the end is given by $$(n_{i})_{i \in I} \mapsto \sum_{i \in I} \alpha_{i} \otimes n_{i} \, .$$
The sum makes sense because only finitely many $n_{i}$'s are nonzero. As a consequence $N^{F}$ has the decomposition \ds{N^{F} = \bigoplus_{i \in I} \left( \alpha_{i} \otimes_{k} N \right)} as a $k$-vector space. Hence the above sum being zero implies that $f_{i}(m) = 0$ for every $i \in I$. Here $m$ was arbitrary, so $f_{i}$'s are zero and hence $\sum \alpha_{i} \otimes f_{i} = 0$.

Now let $g: M^{F} \rarr N^{F}$ be an $A^{F}$-linear map. Since \ds{N^{F} = \bigoplus_{i \in I} \left( \alpha_{i} \otimes_{k} N \right)}, for $m \in M$ we have
\begin{align*}
 g(1 \otimes m) = \sum \alpha_{i} \otimes g_{i}(m)
\end{align*}
where each $g_{i}$ is a function from $M$ to $N$. This representation of $g(1 \otimes m)$ is unique. Therefore via
\begin{align*}
 \sum \alpha_{i} \otimes g_{i}(m + m') 
 &= g(1 \otimes (m + m')) \\
 &= g(1 \otimes m + 1 \otimes m') \\
 &= g(1 \otimes m) + g(1 \otimes m') \\
 &= \sum \alpha_{i} \otimes g_{i}(m) + \sum \alpha_{i} \otimes g_{i}(m') \\
 &= \sum \alpha_{i} \otimes (g_{i}(m) + g_{i}(m')) \, ,
\end{align*}
we conclude that $g_{i}$'s are additive maps. For $r \in A$, we have
\begin{align*}
 g(1 \otimes rm) 
 &= g((1 \otimes r) \cdot (1 \otimes m)) \\
 &=(1 \otimes r) \cdot g(1 \otimes m)
\end{align*}
since $1 \otimes r \in F \otimes_{k} A = A^{F}$ and $g$ is $A^{F}$-linear. From here we get
\begin{align*}
 \sum \alpha_{i} \otimes g_{i}(rm) 
 &= (1 \otimes r) \cdot \left(\sum \alpha_{i} \otimes g_{i}(m) \right) \\
 &= \sum \alpha_{i} \otimes rg_{i}(m) \, ,
\end{align*}
thus $g_{i}$'s are actually $A$-linear. Let $\eta$ be the following composition of $k$-linear maps
\begin{align*}
 \xymatrix{
 M \ar[r]& M^{F} \ar[r]^{g\,\,\,\,\,\,\,\,\,\,\,\,\,\,\,\,\,\,}  & N^{F} \cong \displaystyle{\bigoplus_{i \in I} N}\,
}
\end{align*}
where the first arrow is the natural map $m \mapsto 1 \otimes m$. By our description, $g_{i}$'s are precisely the coordinate maps of $\eta$. Therefore if $\dim_{k} M < \infty$, $\eta$ has finite $k$-rank and hence only finitely many $g_{i}$'s can be nonzero. Thus \ds{f = \sum \alpha_{i} \otimes g_{i}} is a legitimate element of $\left[ \Hom_{A}(M,N)\right]^{F}$. Now since
\begin{align*}
 \theta(f)(1 \otimes m) = \sum \alpha_{i} \otimes g_{i}(m) = g(1 \otimes m)
\end{align*}
for all $m$ and $\theta(f)$, $g$ are $F$-linear maps, we have $\theta(f) = g$. So $\theta$ is surjective.
\end{proof}

\begin{cor} \label{isoAlg}
 If $M$ is a left $A$-module with $\dim_{k} M < \infty$ and $k \subseteq F$ is a field extension, then $\left[ \End_{A}(M) \right]^{F} \cong \End_{A^{F}}(M^{F})$ as $F$-algebras.
\end{cor}
\begin{proof}
 Take $N = M$ in Lemma \ref{iso} and observe that the map $\theta$ preserves multiplication.
\end{proof}

\begin{defn}
 A simple left $A$-module $S$ is called \textbf{absolutely simple} if for every field extension $k \subseteq F$, the left $A^{F}$-module $S^{F}$ is simple.
\end{defn}

\begin{ex}
Every $1$-dimensional left $A$-module is absolutely simple due to dimension reasons, due to Remark \ref{zero}.
\end{ex}

\begin{prop} \label{surj}
 Let $B$ be another $k$-algebra and $\vphi \colon A \twoheadrightarrow B$ a surjective $k$-algebra homomorphism. If $S$ is an absolutely simple left $B$-module, then it is also absolutely simple as a left $A$-module (via $\vphi$).
\end{prop}
\begin{proof}
 As the functor $(-)^{F}$ is exact, the $F$-algebra homomorphism
 \[
  \vphi^{F} \colon A^{F} \to B^{F}
 \]
 is also surjective. The map $\vphi^{F}$ is nothing but the one that gives the $A^{F}$-module structure on $S^{F}$, so it being a surjection yields that the $B^{F}$-submodules and $A^{F}$-submodules of $S^{F}$ coincide. Since $S^{F}$ is simple as a left $B^{F}$-module by our hypothesis, it must also be simple as a left $A^{F}$-module. 
\end{proof}

\begin{thm}[{\cite[Theorem 13.23]{isaacs-algebra} and its proof}] \label{simple-artinian}
 For every division ring $D$ and $n \in \zz_{\geq 1}$, the matrix ring $\mathbb{M}_{n}(D)$ has (up to isomorphism) a unique simple left module $S$ which can be described as the column space $S \coloneqq D^{n}$ acted on via matrix multiplication. Moreover
 \(
  \End_{\mathbb{M}_{n}(D)}(S) \cong D^{\opp}
 \).
\end{thm}

\begin{cor} \label{matrix-abs}
 For every $n \in \zz_{\geq 1}$, the (up to isomorphism) unique simple left module $S$ of the $k$-algebra $\mathbb{M}_{n}(k)$, which can be described as the column space $S \coloneqq k^{n}$ acted on via matrix multiplication, is absolutely simple.
\end{cor}
\begin{proof}
 Let us realize $S$ as the column space $k^{n}$ as we are told in Theorem \ref{simple-artinian}. Then given a field extension $k \subseteq F$, we have $S^{F} \cong F^{n}$ and $[\mathbb{M}_{n}(k)]^{F} \cong \bb{M}_{n}(F)$ (Corollary \ref{isoAlg}) where the action is again matrix multiplication on the column space. Thus $S^{F}$ is simple as a left $[\mathbb{M}_{n}(k)]^{F}$-module again by Theorem \ref{simple-artinian}.
\end{proof}

\begin{thm} \label{abs-simple-defn}
 For a simple left $A$-module $S$ with $\dim_{k} S < \infty$, writing 
\[
 \vphi \colon A \to E \coloneqq \End_{k}(S)
\] 
 for the $k$-algebra homomorphism that defines the $A$-module structure of $S$, the following are equivalent:
\begin{enumerate}
 \item $S$ is absolutely simple.  
 \item For every field extension $k \subseteq F$ with $F$ algebraically closed, the left $A^{F}$-module $S^{F}$ is simple.
 \item $\dim_{k} \End_{A}(S) = 1$.
 \item $\vphi$ is surjective.
\end{enumerate}
\end{thm}
\begin{proof}
 (1) $\Rarr$ (2) is trivial. For (2) $\Rarr$ (3), fix an algebraically closed field $F$ containing $k$, for instance the algebraic closure $\ov{k}$. Now
\begin{align*}
 \dim_{k} \End_{A}(S) 
 &= \dim_{F} \left[
 \End_{A}(S)
 \right]^{F}
 \\
 &= \dim_{F} \End_{A^{F}}(S^{F}) \quad \text{(Corollary \ref{isoAlg})}
 \\
 &= 1 \quad{\text{(Schur's lemma \cite[Theorem 2.1.1]{webb-rep-book})}} \, .
\end{align*}
For (3) $\Rarr$ (4) we want to show that the subset $\vphi(A) \subseteq E$ is \textbf{not} proper. The $k$-algebra $D \coloneqq \End_{A}(S)$ (which is a division algebra by Schur's lemma) essentially by definition is the centralizer $D = \czer{E}{\vphi(A)}$. As 
\[
 \id_{M} \in \czer{E}{E} \subseteq \czer{E}{\vphi(A)} = D \, ,
\]
the assumption $\dim_{k}(D) = 1$ forces $D = \czer{E}{E}$. Pick a finite $k$-basis $X$ of $S$, which is necessarily also a $D$-basis of $S$ as $\dim_{k} D = 1$. Now given $\alpha \in \czer{E}{D}$, by the Jacobson density theorem \cite[Theorem 13.14]{isaacs-algebra} there exists $a \in A$ such that
\begin{align*}
 \vphi(a) |_{X} = \alpha|_{X}
\end{align*}
and hence $\vphi(a) = \alpha$. This shows that $\czer{E}{D} \subseteq \vphi(A)$, hence
\[
 E \subseteq \czer{E}{\czer{E}{E}} = \czer{E}{D} \subseteq \vphi(A)
\]
as desired.

 For (4) $\Rarr$ (1), note that picking any $k$-basis $e_{1}, \dots, e_{n}$ for $S$ establishes isomorphisms $S \cong k^{n}$ and $E \cong \bb{M}_{n}(k)$ where the action is matrix multiplication, so $S$ is absolutely simple as a left $E$-module by Corollary \ref{matrix-abs}. Thus $S$ is absolutely simple as a left $A$-module as well by Proposition \ref{surj}.
\end{proof}

\begin{conv}
From now on, we assume that $A$ is a \textbf{finite-dimensional} $k$-algebra. This in particular ensures that every simple left $A$-module is finite-dimensional over $k$ since it is a quotient of the regular left $A$-module ${_{A}A}$.
\end{conv}


\begin{defn}
 We say (the finite-dimensional $k$-algebra) $A$ is \textbf{split} if every simple left $A$-module is absolutely simple. We call an extension field $F$ of $k$ a \textbf{splitting field} for $A$ if $A^{F}$ is a split $F$-algebra. 
\end{defn}


Observe that by (2) in Theorem \ref{abs-simple-defn}, if $k$ is algebraically closed then $A$ is split. More generally, if $F$ is an algebraically closed field containing $k$ then $F$ is a splitting field for $A$. So splitting fields always exist as we can go to the algebraic closure $\overline{k}$. But $\overline{k}$ is really an overkill and splitting fields can always be found in finite extensions of $k$, which is our next aim to show.

\begin{prop} \label{split-radical}
 Let $k \subseteq F$ be a field extension and use it to identify $(\Rad A)^{F}$ as a left $A^{F}$-submodule of $A^{F}$. If $A$ is split, then $(\Rad A)^{F} = \Rad A^{F}$.
\end{prop}
\begin{proof}
The (left or right) $A^{F}$-module $A^{F} / (\Rad A)^{F} \cong (A / \Rad A)^{F}$ is semisimple since $A$ is split. Hence $\Rad A^{F} \subseteq (\Rad A)^{F}$ by \cite[Lemma 6.3.1, part (2)]{webb-rep-book}. On the other hand, $(\Rad A)^{F} = F \otimes_{k} \Rad A$ is a nilpotent ideal of $A^{F} = F \otimes_{k} A$ so we get the reverse inclusion by \cite[Proposition 6.3.2, part (3c)]{webb-rep-book}.
\end{proof}

We can do this for modules too. Note that given a left $A$-module $U$ and a submodule $V \subseteq U$, we can identify $V^{F}$ as an $A^{F}$-submodule of $U^{F}$.

\begin{cor}
 Suppose $A$ is split and let $U$ be a finite-dimensional left $A$-module. Then identifying $(\Rad U)^{F}$ as a submodule of $U^{F}$, we have $(\Rad U)^{F} = \Rad U^{F}$.
\end{cor}
\begin{proof}
 We have
\begin{align*}
 \Rad U^{F} 
 &= (\Rad A^{F}) \cdot U^{F} \quad \text{\cite[Proposition 6.3.4, part (1c)]{webb-rep-book}}\\
 &= (\Rad A)^{F} \cdot U^{F} \quad \text{(Proposition \ref{split-radical})}\\
 &= (F \otimes_{k} \Rad A) \cdot (F \otimes_{k} U) \\
 &= F \otimes_{k} (\Rad A \cdot U) \\
 &= F \otimes_{k} \Rad U \quad \text{\cite[Proposition 6.3.4, part (1c)]{webb-rep-book}}\\
 &= (\Rad U)^{F} \, .
\end{align*}
\end{proof}

\begin{defn}
 Let $k \subseteq F$ be a field extension. We say $A^{F}$-module $V$ \textbf{can be written in} $k$ if there exists an $A$-module $U$ such that $U^{F} \cong V$.
\end{defn}

\begin{lem} \label{write-in}
 Let $k \subseteq F$ be a field extension and $V$ be an $A^{F}$-module with $\dim_{F} V < \infty$. Then the following are equivalent:
\begin{birki}
 \item $V$ can be written in $k$.
 \item  $V$ has an $F$-basis $v_{1}, \dots, v_{n}$ such that when $\End_{F}(V)$ is identified with $\mm_{n}(F)$ using this basis, the image of the composite map $A \rarr A^{F} \rarr \End_{F}(V)$ lies in $\mm_{n}(k)$. 
\end{birki}
\end{lem}
\begin{proof}
 Assume $V \cong U^{F}$ for some $U$. Since $\dim_{k} U = \dim_{F} V < \infty$, we can pick a $k$-basis $u_{1}, \dots u_{n}$ for $U$ and get a $k$-algebra isomorphism $\End_{k}(U) \cong \mm_{n}(k)$ out of it. Since $(-)^{F}: \alg{k} \rarr \alg{F}$ is the left adjoint of the restriction functor, we have a commutative diagram
\begin{align*}
 \xymatrix{
 A \ar[r] \ar[d] & \End_{k}(U) \ar[d] \\
 A^{F} \ar[r] & \left( \End_{k}(U) \right)^{F} \cong \End_{F}(V)
 }
\end{align*}
of $k$-algebras using the unit of the adjunction. Note that $1 \otimes u_{i}$'s are an $F$-basis for $V$ which gives an isomorphism $\End_{F}(V) \cong \mm_{n}(F)$ and the map which completes

\begin{align*}
 \xymatrix{
 \End_{k}(U) \ar[r] \ar[d] & \mm_{n}(k) \\
 \End_{F}(V) \ar[r] & \mm_{n}(F)
 }
\end{align*}
into a commutative square is just the inclusion $\mm_{n}(k) \emb \mm_{n}(F)$.
Thus the image of the composite map $A \rarr A^{F} \rarr \End_{F}(V)$ lies in $\mm_{n}(k)$ by the commutativity of the first diagram.

Conversely, assume $V$ has such a basis $v_{1}, \dots, v_{n}$. Let $U$ be the $k$-span of these vectors inside $V$. Then by assumption, $U$ is an $A$-submodule of $V$. Now a pure tensor $\alpha \otimes r \in F \otimes_{k} A = A^{F}$ acts on a basis element $v_{i}$ by
\begin{align*}
 (\alpha \otimes r) \cdot v_{i} = (\alpha \cdot (1 \otimes r)) \cdot v_{i}
= \alpha \cdot (rv_{i}) \, .
\end{align*}
Note that $rv_{i} \in U$ here.
On the other hand, $U^{F}$ has an $F$-basis consisting of $1 \otimes u_{i}$'s and we have
\begin{align*}
 (\alpha \otimes r) \cdot (1 \otimes u_{i}) = (\alpha \otimes ru_{i}) = \alpha \cdot (1 \otimes ru_{i}) = \alpha \cdot (r(1 \otimes u_{i}))
\end{align*}
so the bijection $v_{i} \leftrightarrow 1 \otimes u_{i}$ preserves the $A^{F}$-action and $V \cong U^{F}$.
\end{proof}

\begin{thm} \label{chain-ext}
Given field extensions $k \subseteq E \subseteq F$, the following are equivalent:
\begin{birki}
 \item $F$ is a splitting field for $A$ and every simple left $A^{F}$-module can be written in $E$.
 \item $E$ is a splitting field for $A$.
\end{birki}
\end{thm}
\begin{proof}
 (1) $\Rarr$ (2): Let $\mathcal{V}$ be the set of isomorphism classes of simple (and hence absolutely simple) left $A^{F}$-modules. Then there is a set, say $\mathcal{U}$, of isomorphism classes of left $A^{E}$-modules such that 
the map
\begin{align*}
 \mathcal{U} &\to \mathcal{V} 
 \\
 [U] &\mapsto [U^{F}]
\end{align*}
is well-defined and surjective.

Note that whenever $[U] \in \mathcal{U}$ , the left $A^{E}$-module $U$ is simple by Proposition \ref{topdown}. Now let $K$ be an algebraically closed field containing $F$ (we may pick $K = \overline{F}$ for instance). Then whenever $[U] \in \mathcal{U}$, the $A^{K}$-module $U^{K} \cong (U^{F})^{K}$ is simple since $U^{F}$ is absolutely simple. 

Consequently, by part (2) of Theorem \ref{abs-simple-defn} whenever $[U] \in \mathcal{U}$, the left $A^{E}$-module $U$ is absolutely simple. We want to show that $\mathcal{U}$ contains every isomorphism class of simple left $A^{E}$-modules to deduce that $E$ is a splitting field for $A$. 

To that end, let $S$ be a simple left $A^{E}$-module. Then there exists an idempotent $e \in A^{E}$ such that $eS \neq 0$ but $e$ annihilates every other simple left $A^{E}$-module \cite[Theorem 7.3.8]{webb-rep-book}. Since $A^{E}$ embeds into $A^{F}$ we can consider $e$ as a (nonzero) idempotent in $A^{F}$. Thus there exists $[V] \in \mathcal{V}$ such that $eV \neq 0$. Now pick $[U] \in \mathcal{U}$ such that $V \cong U^{F}$. Then $e$ cannot annihilate $U$, which forces $U \cong S$, that is, $[S] = [U] \in \mathcal{U}$.

(2) $\Rightarrow$ (1): We can replace $E$ with $k$ and assume that $A$ is already split. Let $\mathcal{S}$ be a complete list of non-isomorphic simple left $A$-modules. So we have an $A$-linear isomorphism
\begin{align*}
 A / \Rad A \cong \bigoplus_{S \in  \mathcal{S}} S^{\, n_{S}} 
\end{align*}
for some $n_{S}$ (actually $n_{S} = \dim_{k} S$ but we don't need this fact here).  Applying $(-)^{F}$ here and using Proposition \ref{split-radical}, we get an $A^{F}$-linear isomorphism
\begin{align*}
A^{F} / \Rad A^{F} \cong \bigoplus_{S \in \mathcal{S}} \left( S^{F} \right)^{n_{S}} \, .
\end{align*}
Note that each $S^{F}$ is a simple $A^{F}$-module because each $S$ is absolutely simple. Thus every simple $A^{F}$-module must be isomorphic to some $S^{F}$. In other words, simple $A^{F}$-modules can be written in $k$. 

It remains to show that $F$ is a splitting field for $A$, that is, $A^{F}$ is a split $F$-algebra. So let $U$ be a simple left $A^{F}$-module and $F \subseteq K$ a field extension of $F$. By what we just showed, $U \cong S^{F}$ for some simple left $A$-module $S$. Since $A$ is split, $S^{K} = (S^{F})^{K} \cong U^{K}$ is a simple left $A^{K} = (A^{F})^{K}$-module.
\end{proof}

\begin{prop} \label{inter}
 Let $F$ be an algebraic extension of $k$ and $V$ a left $A^{F}$-module with $\dim_{F} V < \infty$. Then there exists an intermediate field $k \subseteq E \subseteq F$ with $[E:k] < \infty$ such that $V$ can be written in $E$.
\end{prop}
\begin{proof}
 Let $v_{1}, \dots, v_{n}$ be an $F$-basis of $V$ with which we can identify $\End_{F}(V)$  with $\mm_{n}(F)$. Pick a $k$-basis $a_{1}, \dots, a_{t}$ for $A$ and let $B_{1}, \dots, B_{t}$ be their images under the composite map $A \rarr A^{F} \rarr \mm_{n}(F)$. Let $E$ be the subfield of $F$ generated by $k$ and the entries of $B_{i}$'s. Since $E$ is generated over $k$ by finitely many algebraic elements, $k \subseteq E$ is a finite extension and by construction $B_{i}$'s lie in $\mm_{n}(E)$. The $a_{i}$'s gets sent to an $E$-basis of $A^{E}$ under the natural map $A \rarr A^{E}$ and hence the composite 
 \[ 
 A^{E} \rarr (A^{E})^{F} = A^{F} \rarr \mm_{n}(F)
 \] 
has image contained in $\mm_{n}(E)$. Invoke Lemma \ref{write-in}.
\end{proof}

\begin{cor} \label{fin-split}
 $A$ has a splitting field which has finite degree over $k$.
\end{cor}
\begin{proof}
 Let $F = \overline{k}$. Being a finite-dimensional $F$-algebra, $A^{F}$ has (up to isomorphism) finitely many simple left modules. So by repeated applications of Proposition \ref{inter}, we get an intermediate field $k \subseteq E \subseteq F$ with $[E:k] < \infty$ such that every simple $A^{F}$-module can be written in $E$. Theorem \ref{chain-ext} applies and $E$ is a splitting field for $A$. 
\end{proof}

\bibliographystyle{hamsalpha}
\bibliography{stable-boy}

\end{document}